\documentclass{article}
\usepackage{amsfonts,amsmath,amsthm,amssymb}
\usepackage{graphics, epsfig}
\usepackage{color}
\usepackage{appendix}
\usepackage{ulem}
\usepackage[makeroom]{cancel}
\usepackage{fancyhdr}

 \usepackage[usenames,dvipsnames]{pstricks}
 \usepackage{pst-grad} 
 \usepackage{pst-plot} 
\allowdisplaybreaks

\let\TeXchi\chi
\newbox\chibox
\setbox0 \hbox{\mathsurround0pt $\TeXchi$}
\setbox\chibox \hbox{\raise\dp0 \box 0 }
\def\chi{\copy\chibox}

\newtheorem{proposition}{Proposition}[section]
\newtheorem{theorem}{Theorem}[section]
\newtheorem{definition}{Definition}[section]
\newtheorem{example}{Example}[section]
\newtheorem{lemma}{Lemma}[section]
\newtheorem{corollary}{Corollary}[section]
\newtheorem{remark}{Remark}[section]

\numberwithin{equation}{section}
\numberwithin{theorem}{section}
\numberwithin{definition}{section}
\numberwithin{example}{section}
\numberwithin{proposition}{section}
\numberwithin{lemma}{section}
\numberwithin{remark}{section}
\begin{document}
\title{On a new space consisting of norm maintaining functions}
\author
{Manuel Norman}
\date{}
\maketitle
\begin{abstract}
\noindent In this paper we define a new space, LH$(X,Y)$, consisting of functions $f\in X \subset Y$ (with $X,Y$ normed spaces) such that $\| f \| _X \equiv \| f \| _Y$ (where $\| \cdot \| _X$ is any norm on $X$, in general not the norm induced by $\| \cdot \| _Y$ on $X$). In Section 2 we study some properties involving LH and Schur's property, norm attainment and LHW (a weaker version of LH). In particular, one of the main results of this section states that strong and weak norm attainments together with some conditions imply that the considered function belongs to LHW or LH (depending on which conditions are satisfied). In Section 3 we renorm in a natural way the space $Y$ so that LH$(X,\widetilde{Y})=X$, obtaining an important extension Theorem.
\end{abstract}
\let\thefootnote\relax\footnote{Author: \textbf{Manuel Norman}; email: manuel.norman02@gmail.com
}
\section{Introduction}
\noindent The principal aim of this paper is to study the space LH(X,Y), some of its properties and its relations with other known notions. We now briefly describe where the idea of this space comes from. First of all, recall that Lip$_0$($S_1$, $S_2$), where $S_1$, $S_2$ are Banach spaces, is the Banach space of all Lipschitz maps $f : S_1 \rightarrow S_2$ such that $f(0)=0$, endowed with the norm:
$$ \|f\|_{Lip_0}:= \sup \lbrace \frac{\|f(x)-f(y)\|_{S_2}}{\|x-y\|_{S_1}}, x \neq y \rbrace $$
(this norm represents the "best" Lipschitz constant for the considered map $f$). For a more detailed discussion on Lip$_0$, see for instance [1-7]. Furthermore, we can also define (for a fixed $\beta >0$) the space H\"ol$^{\beta}_0$($S_1$, $S_2$) (where $S_1$, $S_2$ are Banach) consisting of the $\beta$-H\"older maps $f$ from $S_1$ to $S_2$ such that $f(0)=0$. We can endow this space with the following norm:
$$ \|f\|_{H\ddot{o}l^{\beta}_0}:= \sup \lbrace \frac{\|f(x)-f(y)\|_{S_2}}{\|x-y\|^{\beta}_{S_1}}, x \neq y \rbrace $$
These spaces have many interesting properties: a really important one is norm attainment. The theory of operators that attain their norm has started to be developed many decades ago (see, for instance, [8]); recently, it has been started the study of norm attainment for Lipschitz maps. There are many important kinds of norm attainment in Lip$_0$ (which can actually be defined in a similar way also in H\"ol$^{\beta}_0$); see [1] for a detailed discussion. We will see in Section 2 that some kinds of norm attainment (in a generalised sense) are strictly related to our new space. Because of such important relations, taking inspiration from Lip$_0$ and H\"ol$^{\beta}_0$ we will call our new space LH: L is for Lipschitz, and H for H\"older.\\
The space LH(X,Y) consists, roughly speaking, of the functions in $X\subset Y$ whose norm on X is equal to the norm on Y. A simple example is given by $\mathbb{R}$ and $\mathbb{C}$, considered as normed spaces over $\mathbb{R}$ and with respect to the usual euclidean norm $| \cdot |$. Obviously, we have LH($\mathbb{R},\mathbb{C}$) $=\mathbb{R}$. We can also give another more general example: if $X \subset Y$ is normed under the norm induced by $\| \cdot \|_Y$ on X, then we will certainly have X=LH(X,Y). We will usually consider X not to be defined \textit{a priori} as the space normed under the induced norm, because it is a trivial case in this context. \textit{De facto}, if we prove that under some assumptions LH(X,Y)=X, then it means that the norm on X can be seen as the norm induced by Y on X, but this is because of some conditions that assure it, not by definition. In general, we always have X $\subseteq$ LH(X,Y). We now give:
\begin{definition}\label{Def:1.1}
Let $X$, $Y$ be normed spaces such that $X \subset Y$. A function $f\in X \subset Y$ is a norm maintaining function, and hence belongs to the space of norm maintaining functions LH(X,Y), if $\|  f \|_{X} \equiv \|  f \|_{Y}$.
\end{definition}
In this paper, when we talk about LH(X,Y), we tacitly assume (if not specified otherwise) that $X \subset Y$ are normed spaces (and, as said before, they are in general endowed with different norms; the case with X defined a priori as a normed space having the norm induced by Y will not be considered here) consisting of maps $f:S_1 \rightarrow S_2$, where $S_{1,2}$ are fixed (i.e. all the maps in $X$ and $Y$ are from the same $S_1$ to the same $S_2$) normed spaces.\\
The structure of this article is as follows:\\
$\bullet$ Section 2 is dedicated to some connections of LH to Schur's property, norm attainment (in a generalised sense) and LHW (a weaker version of LH)\\
$\bullet$ in Section 3 we prove one of the main results of this paper: an important extension Theorem, which states that under some conditions we can renorm in a natural way the space $Y$, obtaining $\widetilde{Y}$, so that $X=$ LH(X,$\widetilde{Y}$) (meaning that they have the same elements).
\section{Properties of LH(X,Y) and its relations with other notions}
In the first part of this section we introduce some concepts that will be then used to study some properties of LH.\\
We start by the following generalised version of norm attainment. For a detailed discussion on norm attainment for linear functionals and the different kinds of norm attainment for Lipschitz maps, see [1,8].
\begin{definition}\label{Def:2.1}
Consider a normed space X consisting of functions $T:S_1 \rightarrow S_2$, where $S_1$, $S_2$ are fixed normed spaces (i.e. they are the same for all the maps in X), such that the norm is either ( $\forall T\in X$):
\begin{equation}\label{Eq:2.1}
\| T \|_X := \sup_{(x,y)\in A} \delta_X (x,y,T(x),T(y))
\end{equation}
or
\begin{equation}\label{Eq:2.2}
\| T \|_X := \inf_{(x,y)\in A} \delta_X (x,y,T(x),T(y))
\end{equation}
where $A\subseteq S^2 _1$ is a fixed set and $\delta:A \times (\bigcup_T Im^2 (T))|_A \rightarrow [0,+\infty)$ is a fixed map ($Im (T)$ is the set consisting of all the images of $T$, $Im^2 (T):=Im(T) \times Im(T)$, $\bigcup_T Im^2 (T) \subseteq S^2 _2$ is the union of these couples of images over all the maps $T \in X$, and $(\bigcup_T Im^2 (T))|_A$ indicates that from the union of all these sets we exclude the couples of images corresponding to the couples $(x,y) \not \in A$ (i.e. if $(x,y) \not \in A$, then $(T(x),T(y))$ does not belong to $(\bigcup_T Im^2 (T))|_A$)). We say that a function $f\in X$ strongly attains its X-norm at some $(x_0, y_0) \in A$ if
\begin{equation}\label{Eq:2.3}
\| f \|_X = \delta_X (x_0,y_0,f(x_0),f(y_0))
\end{equation}
\end{definition}
Obviously, the function $\delta_X$ must be such that $\| \cdot \|_X$ is an actual norm.\\
We notice that, for example, the definition of strong norm attainment in Lip$_0$ is a particular case of Definition \ref{Def:2.1}. Indeed, let $A:=\lbrace (x,y) \in S^2 _1, \, x \neq y \rbrace \subset S^2 _1$ and, for $T \in$ Lip$_0(S_1,S_2)$:
$$ \delta_{Lip_0} (x,y,T(x),T(y)):= \frac{\|T(x)-T(y)\|_{S_2}}{\|x-y\|_{S_1}}$$
Then it is clear that 
$$\|T\|_{Lip_0}:= \sup_{(x,y) \in A} \delta_{Lip_0} (x,y,T(x),T(y))$$
and furthermore the definition of strong norm attainment in Lip$_0$ is a particular case of the one given above.\\
Later in this section, we will also deal with the following weaker version of LH:
\begin{definition}\label{Def:2.2}
$LHW(X,Y)$ is the space of maps $f\in X \subset Y$ (X,Y normed spaces) such that:\\
(i) $\| f \|_X \leq \| f \|_Y$\\
(ii) $\exists \lbrace f_n \rbrace$, $f_n \in Y$, such that $\forall \epsilon >0$:
\begin{equation}\label{Eq:2.4}
\| f_n \|_Y < \epsilon + \| f \|_X
\end{equation}
for $n$ enough large.
\end{definition}
We also consider the following weaker version of norm attainment. We notice that the definitions of norm attainment given in [1] can be generalised as we did for strong norm attainment; we will do it later for one case presented in the cited paper (see Definition \ref{Def:2.6}).
\begin{definition}\label{Def:2.3}
Consider a normed space X consisting of functions $T:S_1 \rightarrow S_2$, where $S_1$, $S_2$ are fixed, with norm given by either \eqref{Eq:2.1} or \eqref{Eq:2.2}. Consider a normed space $Y\supset X$ with a norm defined by any of the two said equations, but with the $\sup$ (or the $\inf$) over a set $B$ (where $A\subseteq B \subseteq S^2 _1$), and with $\delta_Y$ instead of $\delta_X$ (they can also be the same). We say that a function $f\in X \subset Y$ weakly attains its (X,Y)-norm at some $(x_0,y_0) \in A$ if $\exists \lbrace f_n \rbrace$, $f_n \in Y$, such that $\forall \epsilon>0$:
\begin{equation}\label{Eq:2.5}
| \| f_n \|_Y - \delta_X (x_0,y_0, f(x_0), f(y_0))|< \epsilon 
\end{equation}
for $n$ enough large.
\end{definition}
Before proving some results related to these concepts, we recall a well known property of normed spaces. In order to do this, we first need the following:
\begin{definition}\label{Def:2.4}
Let X be a topological vector space. We say that a sequence of points $\lbrace x_n \rbrace$ in X converges weakly to $x$, and we write $x_n \rightharpoonup x$, if it converges to $x$ in the weak topology.
\end{definition}
\begin{remark}\label{Rm:2.1}
{\normalfont We note the following important characterisation of weak convergence:}\\
\\
\textit{A net $(x_{\tau})$ in X converges in the weak topology to the element $x$ of X if and only if $\phi(x_{\tau})$ converges to $\phi(x)$ in $\mathbb{R}$ (or $\mathbb{C}$) for all $\phi$ in $X^*$ (the dual space of X).}\\
\\
{\normalfont Hence, a sequence $\lbrace x_n \rbrace$ (which is a particular case of net) converges in the weak topology to $x$ in X if and only if $\forall \phi \in X^*$:
\begin{align*}
\phi(x_{n}) \rightarrow \phi(x)
\end{align*}
where the convergence is here considered as the usual convergence in $\mathbb{R}$ or $\mathbb{C}$.}
\end{remark}
We can now define:
\begin{definition}\label{Def:2.5}
A normed space $X$ has Schur's property if whenever $\lbrace x_n \rbrace$ is a sequence in X such that $x_n \rightharpoonup x \in X$, then $\lim_{n \rightarrow +\infty} \| x_n - x \|_X =0$.
\end{definition}
Finally, we can start to prove some results connecting what we have introduced in this section. The following simple Proposition shows that LHW is indeed a weaker version of LH, because the latter one is always contained in the former.
\begin{proposition}\label{Prop:2.1}
LH(X,Y) $\subseteq$ LHW(X,Y).
\end{proposition}
\begin{proof}
{\normalfont Let $f\in$ LH(X,Y), and take $f_n \equiv f$ $\forall n$. Then we have that $\| f_n \|_Y = \| f \|_Y = \| f \|_X$ and hence $\forall \epsilon >0$:
$$ \|f_n \|_Y = \|f\|_X < \epsilon + \|f \|_X $$
Furthermore, $\|f\|_X \leq \|f \|_Y$ (and actually they are equal) because $f$ belongs to LH(X,Y), and thus $f \in$ LHW(X,Y). Since $f$ was an arbitrary function in LH(X,Y), we conclude that LH(X,Y) $\subseteq$ LHW(X,Y).}
\end{proof}
The following well known result will be used to prove Lemma \ref{Lm:2.2}, which is an important tool for the next proofs.
\begin{lemma}\label{Lm:2.1}
$$ \| f -f _n \|_X \rightarrow 0 \Rightarrow \| f_n \|_X \rightarrow \| f \|_X  $$
\end{lemma}
\begin{proof}
{\normalfont This easily follows from the fact that:
\begin{align*}
| \| x_n\| - \|x\| | \leq \| x_n - x \|
\end{align*}}
\end{proof}
\begin{lemma}\label{Lm:2.2}
Let $f \in$ LHW(X,Y), and suppose that $\exists \lbrace f_n \rbrace$ given by LHW (i.e. $\lbrace f_n \rbrace$, $f_n \in Y$, is a sequence satisfying property (ii) of Definition \ref{Def:2.2}) such that $\| f -f _n \|_Y \rightarrow 0$ as $n \rightarrow +\infty$. Then $f \in$ LH(X,Y).
\end{lemma}
\begin{proof}
{\normalfont By Lemma \ref{Lm:2.1}, we know that $\| f_n \|_Y \rightarrow \|f \|_Y$. This together with the fact that $\forall \epsilon >0$ and for big $n$:
$$ \|f_n\|_Y < \epsilon + \|f\|_X$$
implies that $\forall \epsilon >0$:
$$ \|f\|_Y \leq \epsilon + \|f\|_X$$
which gives: $\|f\|_Y \leq \|f\|_X$. Since $f \in$ LHW(X,Y), $\|f\|_Y \geq \|f\|_X$, and thus $f \in$ LH(X,Y).}
\end{proof}
We now derive an interesting relation between LH and Schur's property. We explicitely note that here with $f_n \rightharpoonup f$ we mean weak convergence with respect to Y.
\begin{proposition}\label{Prop:2.2}
Let $f \in$ LHW(X,Y), where Y has Schur's property. If $\exists \lbrace f_n \rbrace$ given by LHW such that $f_n \rightharpoonup f$, then $f \in$ LH(X,Y).
\end{proposition}
\begin{proof}
{\normalfont Since $f_n \rightharpoonup f$ and Y has Schur's property, $\| f -f _n \|_Y \rightarrow 0$. By applying Lemma \ref{Lm:2.2}, the result follows.}
\end{proof}
This leads to the following important result:
\begin{theorem}\label{Thm:2.1}
Suppose that $\forall f \in$ LHW(X,Y) (where Y has Schur's property) there exist a sequence $\lbrace f_n \rbrace$ given by LHW such that $f_n \rightharpoonup f$. Then LHW(X,Y)=LH(X,Y).
\end{theorem}
\begin{proof}
{\normalfont Use Proposition \ref{Prop:2.1}, Proposition \ref{Prop:2.2} and the fact that the assumptions are on every $f$ in LHW(X,Y).}
\end{proof}
We also state the following Theorem, which also considers (in the second part) a sort of "complementary situation" to LHW.
\begin{theorem}\label{Thm:2.2}
Let $f \in X \subset Y$. If $\exists \lbrace f_n \rbrace$, $f_n \in Y$, such that either:\\
(i) $\forall \epsilon >0$ and for big $n$: $\|f_n\|_Y < \epsilon + \|f\|_X$\\
(ii) $\|T\|_X \leq \|T\|_Y$ $\forall T \in X$\\
(iii) $\|f_n\|_Y \rightarrow \|f\|_Y$ as $n \rightarrow +\infty$\\
or:\\
(i') $\forall \epsilon >0$ and for big $n$: $\|f\|_X < \epsilon + \|f_n\|_Y$\\
(ii') $\|T\|_X \geq \|T\|_Y$ $\forall T \in X$\\
(iii') $\|f_n\|_Y \rightarrow \|f\|_Y$ as $n \rightarrow +\infty$\\
Then $f \in$ LH(X,Y).
\end{theorem}
\begin{proof}
{\normalfont Consider the first case (the proof of the second one is similar). Conditions (i) and (iii) imply that $\|f\|_Y \leq \|f\|_X$. This fact together with condition (ii) gives the result.}
\end{proof}
We now analyse an interesting connection between weak and strong norm attainments:
\begin{proposition}\label{Prop:2.3}
Suppose that $f \in$ LH(X,Y) weakly attains its (X,Y)-norm at some $(x_0,y_0)\in A$ and $\exists \lbrace f_n \rbrace$ given by the definition of weak norm attainment such that $\|f - f_n\|_Y \rightarrow 0$. Then $f$ strongly attains its X-norm at $(x_0,y_0)$.
\end{proposition}
\begin{proof}
{\normalfont By Lemma \ref{Lm:2.1}, we know that $\|f_n\|_Y \rightarrow \|f\|_Y$, and since $f\in$ LH(X,Y) we have that $\|f\|_Y \equiv \|f\|_X$. By definition of weak norm attainment, these imply that $\forall \epsilon >0$:
\begin{align*}
-\epsilon \leq \|f\|_X - \delta_X(x_0,y_0,f(x_0),f(y_0)) \leq \epsilon
\end{align*}
This gives $\|f\|_X = \delta_X(x_0,y_0,f(x_0),f(y_0))$, which concludes the proof.}
\end{proof}
We now prove that weak and strong norm attainments for $f$ together imply, under the condition that $\|f\|_X \leq \|f\|_Y$, that $f$ belongs to LHW. If it is added a certain convergence condition, we obtain that $f$ belongs to LH. This important result is a useful way to prove that a function belongs to LHW or LH, starting from norm attainment.
\begin{theorem}\label{Thm:2.3}
Suppose that $f \in X \subset Y$ strongly attains its X-norm at some $(x_0,y_0) \in A$ and also weakly attains its (X,Y)-norm at the same $(x_0,y_0)$. If $\|f\|_X \leq \|f\|_Y$, then $f \in$ LHW(X,Y). If furthermore $\exists \lbrace f_n \rbrace$ given by the definition of weak norm attainment such that $\|f_n\|_Y \rightarrow \|f\|_Y$, then $f \in $ LH(X,Y).
\end{theorem}
\begin{proof}
{\normalfont We know that $\exists \lbrace f_n \rbrace$ such that $\forall \epsilon >0$ and for $n$ enough large:
$$ | \|f_n\|_Y - \delta_X (x_0,y_0,f(x_0),f(y_0))| < \epsilon $$
By strong norm attainment, we have that $\delta_X (x_0,y_0,f(x_0),f(y_0)) = \|f\|_X$. Therefore, we get:
$$ -\epsilon + \|f\|_X < \|f_n\|_Y < \epsilon + \|f\|_X $$
which, together with the fact that $\|f\|_X \leq \|f\|_Y$, implies that $f \in$ LHW(X,Y). If furthermore there exists a sequence given by the definition of weak norm attainment that satisfies the condition $\|f_n\|_Y \rightarrow \|f\|_Y$, we can conclude (by what we have just proved) that $\forall \epsilon >0$:
$$ -\epsilon + \|f\|_X \leq \|f\|_Y \leq \epsilon + \|f\|_X$$
from which we get: $f \in$ LH(X,Y).}
\end{proof}
\begin{remark}\label{Rm:2.2}
{\normalfont We note that the second conclusion of Theorem \ref{Thm:2.3} does not actually need the assumption that $\|f\|_X \leq \|f\|_Y$, which only assures that $f \in$ LHW(X,Y). We can therefore remove this condition in case we had a sequence $\lbrace f_n \rbrace$ given by the definition of weak norm attainment such that $\|f_n \|_Y \rightarrow \|f\|_Y$. Actually, this inequality will follow from the conclusion itself, since for a function $f\in X \subset Y$ we have: $f \in$ LH(X,Y) $\Leftrightarrow$ $\|f\|_X \leq \|f\|_Y$ $\wedge$ $\|f\|_Y \leq \|f\|_X$ (where $\wedge$ is the logical operator 'and').}
\end{remark}
The following interesting Corollary easily follows (the above Remark still holds):
\begin{corollary}\label{Crl:2.1}
Let X $\subset$ Y be normed spaces, and suppose that whenever $f \in X$, there exists $(x_0, y_0) \in A$ such that $f$ strongly attains its X-norm at $(x_0, y_0)$ and also weakly attains its (X,Y)-norm at $(x_0, y_0)$. Then, if $\|f\|_X \leq \|f\|_Y$ $\forall f \in X$, LHW(X,Y)=X. If furthermore $\forall f \in$ X $\exists \lbrace f_n \rbrace$ given by the definition of weak norm attainment such that $\|f_n\|_Y \rightarrow \|f\|_Y$, then LH(X,Y)=LHW(X,Y)=X. 
\end{corollary}
\begin{proof}
{\normalfont Just use the above Theorem, noting that the conclusions hold $\forall f \in$ X. Obviously, since in general LH(X,Y) $\subseteq$ LHW(X,Y) $\subseteq$ X and in this (second) case we have LH(X,Y)=X, we also have LHW(X,Y)=LH(X,Y)=X.}
\end{proof}
As promised, we now generalise a definition of norm attainment given in [1] and we prove a similar result to Theorem \ref{Thm:2.3}.
\begin{definition}\label{Def:2.6}
Consider a normed space X consisting of functions $T:S_1 \rightarrow S_2$, where $S_1$, $S_2$ are fixed normed spaces, such that the norm is given by either \eqref{Eq:2.1} or \eqref{Eq:2.2}. Suppose that $\delta_X$ is defined in the following way (obviously, it must also be such that $\| \cdot \|_X$ is an actual norm):
\begin{equation}\label{Eq:2.6}
\delta_X(x,y,T(x),T(y)):=\| \widetilde{\delta}_X (x,y,T(x),T(y)) \|_{S_2}
\end{equation}
where $\widetilde{\delta}_X: A \times (\bigcup_T Im^2 (T))|_A \rightarrow S_2$. Then, we say that a function $f\in X$ attains its X-norm towards a point $z \in S_2$ if there exists a sequence $\lbrace (x_n, y_n) \rbrace$, $(x_n,y_n) \in A$, such that:
\begin{equation}\label{Eq:2.7}
\widetilde{\delta}_X (x_n,y_n,f(x_n),f(y_n)) \rightarrow z \quad (in \, S_2)
\end{equation}
and $\|f\|_X=\|z\|_{S_2}$.\\
We say that $f \in X$ weakly attains its (X,Y)-norm towards a point $z \in S_2$ if there exist $\lbrace (x_n, y_n) \rbrace$, $(x_n,y_n) \in A$, and $\lbrace f_n \rbrace$, $f_n \in Y$, such that:
\begin{equation}\label{Eq:2.8}
\widetilde{\delta}_X (x_n,y_n,f(x_n),f(y_n)) \rightarrow z \quad (in \, S_2)
\end{equation}
and $\|f_n\|_Y \rightarrow \|z\|_{S_2}$.
\end{definition}
Again, as for Definition \ref{Def:2.1}, we notice that the definition in Lip$_0$ of norm attainment towards a point is a particular case of Definition \ref{Def:2.6}.\\
When we deal with norm attainment towards a point, we will always tacitly assume that $\delta_X(x,y,T(x),T(y)):=\| \widetilde{\delta}_X (x,y,T(x),T(y))\|_{S_2}$ (this note is important in particular if we consider at the same time several kinds of norm attainment).\\
The following Proposition can be easily proved:
\begin{proposition}\label{Prop:2.4}
Suppose that $f \in$ LH(X,Y) attains its X-norm towards some $z\in S_2$. Then $f$ weakly attains its (X,Y)-norm towards the same $z$.
\end{proposition}
\begin{proof}
Take $f_n \equiv f$ $\forall n$. By strong norm attainment towards $z$, and by the fact that $f$ belongs to LH(X,Y), we have: $\|f_n\|_Y = \|f\|_Y  = \|f\|_X = \|z\|_{S_2}$. Thus we can say that $\|f_n\|_Y \rightarrow \|z\|_{S_2}$, and since we can take the same sequence $\lbrace (x_n,y_n) \rbrace$ given by X-norm attainment towards $z$, the conclusion follows.
\end{proof}
Furthermore, there is an important link between strong norm attainment and norm attainment towards a point. Indeed, the space consisting of all the maps that strongly attain their X-norm at some point is included in the space of the maps that attain their X-norm towards some point (and similarly for weak norm attainment). See also the discussion before Proposition \ref{Prop:2.4}.
\begin{proposition}\label{Prop:2.5}
If $f\in X$ strongly attains its X-norm at some point $(x_0,y_0)\in A$, then it attains its X-norm towards $\widetilde{\delta}_X (x_0,y_0,f(x_0),f(y_0))$. The result also holds if we consider $f \in X \subset Y$ and we use weak norm attainment and weak norm attainment towards $\widetilde{\delta}_X (x_0,y_0,f(x_0),f(y_0))$.
\end{proposition}
\begin{proof}
Take the constant sequence $\lbrace (x_0,y_0) \rbrace$. We know that (as already said, we assume $\delta_X(x,y,T(x),T(y)):=\| \widetilde{\delta}_X (x,y,T(x),T(y))\|_{S_2}$ because we are dealing also with weak norm attainment towards a point):
$$\|f\|_X=\delta_X(x_0,y_0,f(x_0),f(y_0))=\| \widetilde{\delta}_X (x_0,y_0,f(x_0),f(y_0)) \|_{S_2}$$
Since obviously $\widetilde{\delta}_X (x_0,y_0,f(x_0),f(y_0))$ tends to itself, by the first and last terms of the above equality we can conclude that $f$ attains its X-norm towards $\widetilde{\delta}_X (x_0,y_0,f(x_0),f(y_0))$.\\
For the second part of the Proposition, first note that, $\forall \epsilon >0$ and for big $n$:
$$ | \|f_n\|_Y - \delta_X(x_0,y_0,f(x_0),f(y_0))|=|\|f_n\|_Y - \| \widetilde{\delta}_X(x_0,y_0,f(x_0),f(y_0))\|_{S_2} | < \epsilon $$
Then obviously $\|f_n\|_Y \rightarrow \| \widetilde{\delta}_X(x_0,y_0,f(x_0),f(y_0))\|_{S_2}$, and since $\widetilde{\delta}_X(x_0,y_0,f(x_0),f(y_0))$ tends to itself, the Proposition follows by taking the constant sequence $\lbrace (x_0,y_0) \rbrace$.
\end{proof}
We now prove a basic result (of course, $\delta_X(x,y,T(x),T(y)):=$ \\
$\| \widetilde{\delta}_X (x,y,T(x),T(y))\|_{S_2}$ and $\delta_Y(x,y,T(x),T(y)):=\| \widetilde{\delta}_Y (x,y,T(x),T(y))\|_{S_2}$):
\begin{proposition}\label{Prop:2.6}
If $f\in X \subset Y$ attains its X-norm towards some $z \in S_2$ and also attains its Y-norm towards the same $z$, then $f \in$ LH(X,Y).
\end{proposition}
\begin{proof}
By definition of norm attainment towards $z$, we have that $\|f\|_X=\|z\|_{S_2}$, and $\|f\|_Y=\|z\|_{S_2}$. The Proposition follows.
\end{proof}
The following Proposition gives an interesting way to prove that there exist sequences such that $\|f_n\|_Y \rightarrow \|f\|_Y$:
\begin{proposition}\label{Prop:2.7}
Suppose that $f\in X \subset Y$ attains its Y-norm towards some $z \in S_2$ and also weakly attains its (X,Y)-norm towards the same $z$. Then every sequence $\lbrace f_n \rbrace $ given by the definition of weak norm attainment towards $z$ is such that $\|f_n\|_Y \rightarrow \|f\|_Y$.
\end{proposition}
\begin{proof}
By definition of weak norm attainment towards $z$ we know that, whenever $\lbrace f_n \rbrace$ is given by weak norm attainment towards $z$, $\|f_n\|_Y \rightarrow \|z\|_{S_2}$. By norm attainment towards the same $z$, we have that $\|f\|_Y=\|z\|_{S_2}$. The conclusion easily follows.
\end{proof}
We note that this Proposition is useful in particular because it assures that \textit{every} sequence given by weak norm attainment towards $z$ is such that its norm on Y tends to the norm of $f$ on Y.\\
We now state an intersting result involving LHW and norm attainment towards a point:
\begin{proposition}\label{Prop:2.8}
Let $f\in$ LHW(X,Y), and suppose that $f$ attains its Y-norm towards some $z \in S_2$ and also weakly attains its (X,Y)-norm towards the same $z$. If there exists at least one sequence $\lbrace f_n \rbrace$ given by both LHW and weak norm attainment towards $z$, then $f \in$ LH(X,Y).
\end{proposition}
\begin{proof}
By the definition of LHW, we have that $\|f\|_X \leq \|f\|_Y$. Furthermore, we know that there is a sequence given by both the said properties (i.e. the sequence is the same and satisfies the conditions of the two properties), say $\lbrace f_n \rbrace$, with $f_n \in Y$, such that, $\forall \epsilon >0$ and for big $n$:
$$ \|f_n \|_Y < \epsilon + \|f\|_X $$
Since $\lbrace f_n \rbrace$ is given both by LHW and by weak norm attainment towards $z$, we can use this inequality together with Proposition \ref{Prop:2.7} (which assures us that every sequence given by weak norm attainment, under the considered assumptions, is such that its norm on Y converges to the norm of $f$ on Y) to get: 
$$\|f\|_Y \leq \|f\|_X$$
This, together with $\|f\|_X \leq \|f\|_Y$, implies that $f \in$ LH(X,Y).
\end{proof}
We conclude this section with the following important Theorem, which gives another interesting way to prove that a function belongs to LH:
\begin{theorem}\label{Thm:2.4}
Suppose that $f \in X \subset Y$ attains its X-norm towards $z \in S_2$, and it also weakly attains its (X,Y)-norm towards the same $z$. If $\exists$ $\lbrace f_n \rbrace$ given by the definition of weak norm attainment towards $z$ such that $\|f- f_n\|_Y \rightarrow 0$, then $f \in$ LH(X,Y).
\end{theorem}
\begin{proof}
We know by weak norm attainment towards $z$ that $\exists$ $\lbrace f_n \rbrace$ such that $\|f_n\|_Y \rightarrow \|z\|_{S_2}=\|f\|_X$ (where the last equality follows from norm attainment towards the same $z$). Since there is a sequence given by weak norm attainment such that $\|f- f_n\|_Y \rightarrow 0$, by Lemma \ref{Lm:2.1} we know that $\|f_n\|_Y \rightarrow \|f\|_Y$. But $\lim_{n\rightarrow +\infty} \|f_n\|_Y$ must be unique, and hence $\|f\|_X = \|f\|_Y$, from which follows that $f \in$ LH(X,Y).
\end{proof}
\section{Extension Theorem for LH(X,Y)}
It is natural to ask whether it is possible or not to renorm Y so that X has the same elements of LH(X,Y). We will do this by starting from the norm defined on X, and then extending it to the space Y in a natural way. We first recall some important definitions and results. We mainly refer to [11]; see also [12-17].
\begin{definition}\label{Def:3.1}
Consider a normed space Y and a nonempty subset X of Y. We define the following set valued map $P_X:Y\rightarrow \mathcal{P}(Y)$:
\begin{equation}\label{Eq:3.1}
P_X(y):=\lbrace x \in X : \| y-x\|_Y = d(y,X) \rbrace
\end{equation}
We say that the elements $x$ in $P_X(y)$ are the best approximations (or the nearest points) to $y\in Y$. We call X proximinal if $P_X(y) \neq \emptyset$ $\forall y\in Y$. We call X finite proximinal if it is proximinal and $P_X(y)$ has a finite number of elements $\forall y \in Y$. We call X Chebyshev if it is proximinal and $P_X(y)$ is a singleton set $\forall y \in Y$. Notice that every Chebyshev set is finite proximinal.
\end{definition}
\begin{remark}\label{Rm:3.1}
\normalfont{Notice that: $x \in P_X(y)$ if and only if $\| y-x \|_Y \leq \|y-z\|_Y $ $\forall z \in X$.}
\end{remark}
The main Theorem of this section involves proximinal and Chebyshev spaces, so we state some results that give some examples of these kinds of sets. We briefly recall some important properties; more detailed discussions about them can be found in almost any text on Banach space theory (see, for instance, [9,10]). A subset $S$ of a vector space $V$ is convex if for all $x,y \in S$ and $t \in [0,1]$, the linear (actually affine) combination $(1-t)x + y$ belongs to $S$. A strictly convex space X is a normed space such that:
$$ x \neq y \wedge \|x\|=\|y\|=1 \Rightarrow \|x+y\| < 2 $$
A space is called uniformly convex if for every $\epsilon \in (0,2] $ there is $\delta >0$ such that, for any $x,y$ $|$ $\|x\| \leq 1$, $\|y\| \leq 1$ and $\|x-y\| \geq \epsilon$, one has $\|\frac{x+y}{2}\| \leq 1 - \delta$. A Banach space X is reflexive if every continuous linear functional on X attains its maximum on the closed unit ball in X (this characterisation of reflexivity is known as James' Theorem; for a discussion on James'-type results, see for instance [18]).
\begin{proposition}\label{Prop:3.1}
Every nonempty, closed, convex subset of a reflexive, strictly convex Banach space is a Chebyshev set.
\end{proposition}
\begin{proof}
See Theorem 2.4.14 in [11].
\end{proof}
\begin{corollary}\label{Crl:3.1}
Every nonempty, closed, convex subset of a uniformly convex Banach space is a Chebyshev set.
\end{corollary}
\begin{corollary}\label{Crl:3.2}
Any nonempty, closed, convex subset of a Hilbert space or of $L^p(\mu)$, equipped with the usual p-norm and with $1 < p < +\infty$, is a Chebyshev set.
\end{corollary}
\begin{proposition}\label{Prop:3.2}
Any nonempty, closed, convex subset of a reflexive Banach space is proximinal.
\end{proposition}
\begin{proof}
See Proposition 2.4.13 in [11].
\end{proof}
\begin{definition}\label{Def:3.2}
Let $(\mathcal{M},\mathcal{A}, \mu)$ be a finite measure space. For a Banach space Y consider (for $1 \leq p < +\infty$) the Banach space of Bochner p-integrable (equivalence classes of) functions on $\mathcal{M}$ with values in Y, and (for $p= +\infty$) the space of essentially bounded (equivalence classes of) functions on $\mathcal{M}$ with values in Y, endowed with the usual p-norm (these spaces are indicated by $L^p(\mathcal{M},Y)$):
\begin{equation}\label{Eq:3.2}
\| f \|_{L^p(\mathcal{M},Y)}:=(\int_{\mathcal{M}} \| f(t) \|^p _Y d \mu(t) )^{1/p}
\end{equation}
for $1 \leq p < +\infty$, and
\begin{equation}\label{Eq:3.3}
\| f \|_{L^\infty (\mathcal{M},Y)}:= \mathop{\mathrm{ess~sup}}\limits_{t \in \mathcal{M}}  \|f(t) \|_Y
\end{equation}
Note that $L^p (\mu)$ is a particular case of this space (just take Y$=\mathbb{R}$, with the usual Euclidean norm $| \cdot |$, and similarly for the complex case).
\end{definition}
\begin{proposition}\label{Prop:3.3}
Let X be a separable, closed, convex subspace of a Banach space Y. Then: $L^p (\mathcal{M},X)$ (separable) is proximinal in $L^p (\mathcal{M},Y)$ if and only if X is proximinal in Y ($1 \leq p \leq +\infty$).
\end{proposition}
\begin{proof}
See point b in Theorem 1.1 in [12], and the slightly different definition of proximinality given there (i.e. X proximinal in Y is assumed to be closed and convex, and Y is considered Banach. In our definition, which is the same given in [11], we do not require in general these conditions). See this paper for some other results of this kind.
\end{proof}
We now define two properties that will be used in our main result:
\begin{definition}\label{Def:3.3}
A normed proximinal space X in Y is said to be triangular if:
\begin{equation}\label{Eq:3.4}
\sup_{x \in P_X(y_1 + y_2)} \|x\|_X \leq \sup_{x_1 \in P_X(y_1)} \|x_1 \|_X + \sup_{x_2 \in P_X(y_2)} \|x_2 \|_X
\end{equation}
for every $y_1, y_2 \in Y$ (the suprema are allowed to be infinite, and even in such cases they must satisfy the above inequality).\\
Notice that a normed Chebyshev space is triangular when ($x$ is the unique element in $P_X(y_1 + y_2)$, and $x_1$, $x_2$ are the unique elements in $P_X(y_1)$, $P_X(y_2)$ respectively):
\begin{equation}\label{Eq:3.5}
\|x\|_X \leq  \|x_1 \|_X + \|x_2 \|_X
\end{equation}
for all $y_1, y_2 \in Y$.\\
A normed proximinal space X in Y is bounded if:
\begin{equation}\label{Eq:3.6}
\sup_{x \in P_X(y)} \|x\|_X < +\infty
\end{equation}
$\forall y \in Y$.\\
We notice that a normed Chebyshev space is always bounded proximinal, because $\forall y \in Y$ the unique element in $P_X(y)$ has finite norm (by definition of norm).\\
When we say that a space is triangular proximinal or bounded proximinal, we always tacitly assume that it is normed.
\end{definition}
\begin{example}\label{Example:3.1}
\normalfont{We give a really basic example of bounded triangular proximinal (actually Chebyshev) space. Consider the two spaces $\mathbb{R}$ and $\mathbb{C}$, normed under the usual Euclidean norm $| \cdot |$ over $\mathbb{R}$. We first verify that $\mathbb{R}$ is Chebyshev (and hence bounded proximinal) in $\mathbb{C}$. Consider, for $z \in \mathbb{C}$:}
$$|z-r| \leq |z-r'|$$
\normalfont{for all $r' \in \mathbb{R}$ and for some $r \in \mathbb{R}$. Since:}
$$ |z-r'|=\sqrt{(\Re(z)-r')^2 + \Im^2(z)} $$
\normalfont{if we take $r=\Re(z)$ we have:}
$$ |z-r|=| \Im(z) | \leq \sqrt{(\Re(z)-r')^2 + \Im^2(z)} = |z-r'|, \quad \forall r' \in \mathbb{R} $$
\normalfont{It easily follows that $P_{\mathbb{R}}(z)=\lbrace \Re(z) \rbrace$, whichever is $z \in \mathbb{C}$. Thus $\mathbb{R}$ is Chebyshev in $\mathbb{C}$. It is also clear, by triangle inequality (and by the fact that $\Re(z_1+z_2)=\Re(z_1)+\Re(z_2)$), that:}
$$ | \Re(z_1) + \Re(z_2) | \leq |\Re(z_1)| + |\Re(z_2)| $$
\normalfont{and hence $\mathbb{R}$ is also triangular.\\
It is important to notice that actually $\mathbb{R}$ could be normed with any other norm: indeed, to prove that $\mathbb{R}$ is Chebyshev in $\mathbb{C}$ we do not need to define any norm on the first space; we only need to do it when we want to prove that $\mathbb{R}$ is triangular, but since for any norm on $\mathbb{R}$ we certainly have (by triangular inequality of the norm):
$$ \|\Re(z_1)+\Re(z_2)\|_ {\mathbb{R}} \leq \|\Re(z_1)\|_{\mathbb{R}}+ \|\Re(z_2)\|_{\mathbb{R}}$$
we can conclude that:}
\begin{proposition}\label{Prop:3.4}
$\mathbb{R}$ is always triangular Chebyshev in $\mathbb{C}$ (where $\mathbb{C}$ is normed over $\mathbb{R}$ under the Euclidean norm $| \cdot |$), for every norm on $\mathbb{R}$.
\end{proposition}
\end{example}
We can finally state and prove the main Theorem of this Section:
\begin{theorem}\label{Thm:3.1}
Let X be a nonempty subset of a normed space Y. Suppose that X is a normed space under any $\| \cdot \|_X$ (in general different from the norm induced by $\| \cdot \|_Y$). If X is a triangular bounded proximinal space in Y, then the following is a seminorm on Y:
\begin{equation}\label{Eq:3.7}
\| f \|_{\widetilde Y} := \sup_{g \in P_X(f)} \|g\|_X
\end{equation}
and it is a norm on the quotient space $\widetilde Y:= Y^{\| \cdot \|}/\mathcal{N}$, where $Y^{\| \cdot \|}$ is Y under the above seminorm, and
$$ \mathcal{N}:= \lbrace g \in Y : \| g \|_{\widetilde Y}=0 \rbrace $$
Furthermore, we have that X=LH(X,$\widetilde Y$).
\end{theorem}
\begin{proof}
We start by proving that $\| \cdot \|_{\widetilde Y}$ is a seminorm on $Y$, and thus a norm on $\widetilde Y$. First of all, notice that $\| \cdot \|_{\widetilde Y} : Y \rightarrow [0,+\infty)$ (it cannot be $+\infty$ because X is bounded proximinal). Moreover, $f \equiv 0$ implies that ($g \in X$):
$$ \forall h \in X: \|f-g\|_Y \leq \|f-h\|_Y \Rightarrow \forall h \in X: \|g\|_Y \leq \|h\|_Y $$
and the unique element $g$ in X satisfying this property is $g \equiv 0$. Hence:
$$\| 0 \|_{\widetilde Y} = \|0\|_X = 0$$
The triangle inequality:
$$ \|f_1 + f_2\|_{\widetilde Y} \leq \|f_1\|_{\widetilde Y} + \|f_2\|_{\widetilde Y}$$
is satisfied $\forall f_1, f_2 \in Y$ because X is triangular. We now turn to the last property:
$$ \|a f\|_{\widetilde Y}=|a| \|f\|_{\widetilde Y}$$
whenever $f \in Y$, $a \in \mathbb{R}$. When $a=0$, this obviously holds. For the other cases, we have by definition:
$$\|a f\|_{\widetilde Y} = \sup_{g \in P_X(af)} \|g\|_X$$
As noted in Remark \ref{Rm:3.1}, $g \in P_X(af)$ $\Leftrightarrow$ $\| af-g \|_Y \leq \|af-h\|_Y $ $\forall h \in X$. The second condition can be rephrased in the following way:
$$ \| f-\frac{g}{a} \|_Y \leq \|f-\frac{h}{a}\|_Y ,  \forall h \in X $$
But since we know that $\frac{h}{a}$ is any element of X because it is a (vector) normed space, and the best approximations to $f$ are given in $P_X(f)$, we can conclude that $\frac{g}{a}$ (which is by the above inequality a best approximation to $f$) coincides with an element of $P_X(f)$, say $\widetilde g \equiv \frac{g}{a}$. Therefore, we have that:
\begin{equation}\label{Eq:3.8}
P_X(af)=a P_X(f)
\end{equation}
from which follows that
$$\|a f\|_{\widetilde Y}=|a| \|f\|_{\widetilde Y}$$
Thus, $\| \cdot \|_{\widetilde Y}$ is a seminorm on Y. It is not necessarily a norm, because $\|f\|_{\widetilde Y}=0$ implies that $\sup_{g \in P_X(f)} \|g\|_X =0$, and hence that $g \equiv 0$, but this does not necessarily imply that $f \equiv 0$.\\
We now verify that $\|f\|_X \equiv \|f\|_{\widetilde Y}$ $\forall f \in X$. It is clear that, if $f \in X$, then $P_X(f)= \lbrace f \rbrace$ (this follows from Remark \ref{Rm:3.1} and from the fact that $\|f-g\|_Y=0 \Leftrightarrow f \equiv g$). Therefore:
$$\|f\|_{\widetilde Y}=\sup_{g\in P_X(f)} \|g\|_X = \|f\|_X$$
We now only need to verify that $X \subset \widetilde Y$. Since obviously $0 \in \widetilde Y$, we have to check that there is no $f \in X \setminus \lbrace 0 \rbrace$ : $\|f\|_{\widetilde Y} =0$. But for what we have just proved, this is the same as $\|f\|_X$, which is a norm and hence is equal to $0$ if and only if $f \equiv 0$. Thus, $X\subset \widetilde Y$. We can therefore conclude that X=LH(X,$\widetilde Y$).
\end{proof}
\begin{remark}\label{Rm:3.2}
\normalfont{We explicitely notice that, since a (normed) Chebyshev space is bounded proximinal, the above result also holds when X is (normed) triangular Chebyshev.}
\end{remark}
\section{Conclusion}
In this paper we have introduced a new interesting function space consisting of 'norm maintaining functions'. We have then found an important extension Theorem (Theorem \ref{Thm:3.1}) and we have also proved some interesting results connecting this new space with other known concepts (and in particular with norm attainment). We think that this work can be continued in at least two directions: by finding other extension theorems (using different renormings of Y so that X=LH(X,Y)) and by studying the various relations between LH (and LHW) and all the (generalised) notions of norm attainment (some of them given, for instance, in [1]). \\
\\
\begin{large}
\textbf{Acknowledgements}
\end{large}\\
\\
I thank Miguel Mart\'{i}n for his useful comments on the first version of this paper.\\
\\
\begin{large}
\textbf{References}
\end{large}\\
\\
\begin{small}
$[1]$ G. Choi, Y. S. Choi, M. Mart\'{i}n. \textit{Emerging notions of norm attainment for Lipschitz maps between Banach spaces}. J. Math. Anal. Appl. Volume 483, Issue 1 (2020)\\
$[2]$ B. Cascales, R. Chiclana, L. Garc\'{i}a-Lirola, M. Mart\'{i}n, A. Rueda Zoca. \textit{On strongly norm attaining Lipschitz maps}. J. Funct. Anal. 277 (2019), no. 6, 1677-1717\\
$[3]$ V. Kadets, M. Mart\'{i}n, M. Soloviova. \textit{Norm-attaining Lipschitz functionals}. Banach J. Math. Anal. Volume 10, Number 3 (2016), 621-637\\
$[4]$ A. Dalet, G. Lancien. \textit{Some properties of coarse Lipschitz maps between Banach spaces}. North-West. Eur. J. Math. 3 (2017), 41-62 \\
$[5]$ G. Godefroy. \textit{A survey on Lipschitz-free Banach spaces}. Comment. Math. 55 (2015) 89-118\\
$[6]$ G. Godefroy. \textit{On norm attaining Lipschitz maps between Banach spaces}. Pure Appl. Funct. Anal. 1 (2016), 39-46\\
$[7]$  M. D. Acosta. \textit{Denseness of norm attaining mappings}. RACSAM 100 (2006), 9-30\\
$[8]$ J. Lindenstrauss. \textit{On operators which attain their norm}. Israel J. Math. 1 (1963), 139-148\\
$[9]$ R. E. Megginson. \textit{An Introduction to Banach Space Theory}. New York Berlin Heidelberg: Springer-Verlag (1998)\\
$[10]$ F. Albiac, N. J. Kalton. \textit{Topics in Banach space theory}, Second Edition, Graduate Texts
in Mathematics. Vol. 233, Springer, New York (2016)\\
$[11]$ J. Fletcher. \textit{The Chebyshev Set Problem}. Master's Thesis, University of Auckland (2013)\\
$[12]$ T. Paul. \textit{Various notions of best approximation property in spaces of Bochner integrable functions}. Adv. Oper. Theory. Volume 2, Number 1 (2017), 59-77\\
$[13]$ R. Smarzewski. \textit{Strongly unique best approximation in Banach spaces}. J. Approx. Theory, Volume 47, Issue 3 (1986), 184-194\\
$[14]$ M. Mart\'{i}n. \textit{On proximinality of subspaces and the lineability of the set of norm-attaining functionals of Banach spaces}. J. Funct. Anal. (2019)\\
$[15]$ T. W. Narang, S. Gupta. \textit{Proximinality and co-proximinality in metric linear spaces}. Ann. Univ. Mariae Curie-Sklodowska Sect. A. Volume 69, Number 1 (2015), 83-90\\
$[16]$ E. W. Cheney. \textit{Introduction to approximation theory}, 2nd ed. AMS publishing, Providence (1982)\\
$[17]$ I. Singer. \textit{The theory of best approximation and functional analysis}. Volume 13 of Series in applied mathematics. SIAM, Philadelphia (1974)\\
$[18]$ M. D. Acosta, J. Becerra-Guerrero, M. Ruiz-Galán. \textit{Characterizations of the
reflexive spaces in the spirit of James’ Theorem}. Cont. Math. 3-21 (2003)
\end{small}
\end{document}